\UseRawInputEncoding
\documentclass[oneside, a4paper,reqno]{amsart}
\usepackage{pdfsync}
\usepackage{stmaryrd}
\usepackage{mathrsfs}
\usepackage{multicol}
\usepackage{amsmath, amsthm, amscd, amssymb, latexsym, eucal}
\usepackage[all]{xy}
\def\serieslogo@{} \def\@setcopyright{} \makeatother
\usepackage{multienum}

\usepackage{hyperref}
\usepackage{color}
\usepackage{cite}

\hypersetup{colorlinks=true,
     breaklinks=true,
     linkcolor=blue,    
     bookmarks=true,
      pageanchor=true}

\makeatletter
\renewcommand*\env@matrix[1][c]{\hskip -\arraycolsep
  \let\@ifnextchar\new@ifnextchar
  \array{*\c@MaxMatrixCols #1}}
\makeatother

\usepackage{color}

\definecolor{AliceBlue}{rgb}{0.94,0.97,1.00}
\definecolor{AntiqueWhite1}{rgb}{1.00,0.94,0.86}
\definecolor{AntiqueWhite2}{rgb}{0.93,0.87,0.80}
\definecolor{AntiqueWhite3}{rgb}{0.80,0.75,0.69}
\definecolor{AntiqueWhite4}{rgb}{0.55,0.51,0.47}
\definecolor{AntiqueWhite}{rgb}{0.98,0.92,0.84}
\definecolor{BlanchedAlmond}{rgb}{1.00,0.92,0.80}
\definecolor{BlueViolet}{rgb}{0.54,0.17,0.89}
\definecolor{CadetBlue1}{rgb}{0.60,0.96,1.00}
\definecolor{CadetBlue2}{rgb}{0.56,0.90,0.93}
\definecolor{CadetBlue3}{rgb}{0.48,0.77,0.80}
\definecolor{CadetBlue4}{rgb}{0.33,0.53,0.55}
\definecolor{CadetBlue}{rgb}{0.37,0.62,0.63}
\definecolor{CornflowerBlue}{rgb}{0.39,0.58,0.93}
\definecolor{DarkBlue}{rgb}{0.00,0.00,0.55}
\definecolor{DarkCyan}{rgb}{0.00,0.55,0.55}
\definecolor{DarkGoldenrod1}{rgb}{1.00,0.73,0.06}
\definecolor{DarkGoldenrod2}{rgb}{0.93,0.68,0.05}
\definecolor{DarkGoldenrod3}{rgb}{0.80,0.58,0.05}
\definecolor{DarkGoldenrod4}{rgb}{0.55,0.40,0.03}
\definecolor{DarkGoldenrod}{rgb}{0.72,0.53,0.04}
\definecolor{DarkGray}{rgb}{0.66,0.66,0.66}
\definecolor{DarkGreen}{rgb}{0.00,0.39,0.00}
\definecolor{DarkGrey}{rgb}{0.66,0.66,0.66}
\definecolor{DarkKhaki}{rgb}{0.74,0.72,0.42}
\definecolor{DarkMagenta}{rgb}{0.55,0.00,0.55}
\definecolor{DarkOliveGreen1}{rgb}{0.79,1.00,0.44}
\definecolor{DarkOliveGreen2}{rgb}{0.74,0.93,0.41}
\definecolor{DarkOliveGreen3}{rgb}{0.64,0.80,0.35}
\definecolor{DarkOliveGreen4}{rgb}{0.43,0.55,0.24}
\definecolor{DarkOliveGreen}{rgb}{0.33,0.42,0.18}
\definecolor{DarkOrange1}{rgb}{1.00,0.50,0.00}
\definecolor{DarkOrange2}{rgb}{0.93,0.46,0.00}
\definecolor{DarkOrange3}{rgb}{0.80,0.40,0.00}
\definecolor{DarkOrange4}{rgb}{0.55,0.27,0.00}
\definecolor{DarkOrange}{rgb}{1.00,0.55,0.00}
\definecolor{DarkOrchid1}{rgb}{0.75,0.24,1.00}
\definecolor{DarkOrchid2}{rgb}{0.70,0.23,0.93}
\definecolor{DarkOrchid3}{rgb}{0.60,0.20,0.80}
\definecolor{DarkOrchid4}{rgb}{0.41,0.13,0.55}
\definecolor{DarkOrchid}{rgb}{0.60,0.20,0.80}
\definecolor{DarkRed}{rgb}{0.55,0.00,0.00}
\definecolor{DarkSalmon}{rgb}{0.91,0.59,0.48}
\definecolor{DarkSeaGreen1}{rgb}{0.76,1.00,0.76}
\definecolor{DarkSeaGreen2}{rgb}{0.71,0.93,0.71}
\definecolor{DarkSeaGreen3}{rgb}{0.61,0.80,0.61}
\definecolor{DarkSeaGreen4}{rgb}{0.41,0.55,0.41}
\definecolor{DarkSeaGreen}{rgb}{0.56,0.74,0.56}
\definecolor{DarkSlateBlue}{rgb}{0.28,0.24,0.55}
\definecolor{DarkSlateGray1}{rgb}{0.59,1.00,1.00}
\definecolor{DarkSlateGray2}{rgb}{0.55,0.93,0.93}
\definecolor{DarkSlateGray3}{rgb}{0.47,0.80,0.80}
\definecolor{DarkSlateGray4}{rgb}{0.32,0.55,0.55}
\definecolor{DarkSlateGray}{rgb}{0.18,0.31,0.31}
\definecolor{DarkSlateGrey}{rgb}{0.18,0.31,0.31}
\definecolor{DarkTurquoise}{rgb}{0.00,0.81,0.82}
\definecolor{DarkViolet}{rgb}{0.58,0.00,0.83}
\definecolor{DeepPink1}{rgb}{1.00,0.08,0.58}
\definecolor{DeepPink2}{rgb}{0.93,0.07,0.54}
\definecolor{DeepPink3}{rgb}{0.80,0.06,0.46}
\definecolor{DeepPink4}{rgb}{0.55,0.04,0.31}
\definecolor{DeepPink}{rgb}{1.00,0.08,0.58}
\definecolor{DeepSkyBlue1}{rgb}{0.00,0.75,1.00}
\definecolor{DeepSkyBlue2}{rgb}{0.00,0.70,0.93}
\definecolor{DeepSkyBlue3}{rgb}{0.00,0.60,0.80}
\definecolor{DeepSkyBlue4}{rgb}{0.00,0.41,0.55}
\definecolor{DeepSkyBlue}{rgb}{0.00,0.75,1.00}
\definecolor{DimGray}{rgb}{0.41,0.41,0.41}
\definecolor{DimGrey}{rgb}{0.41,0.41,0.41}
\definecolor{DodgerBlue1}{rgb}{0.12,0.56,1.00}
\definecolor{DodgerBlue2}{rgb}{0.11,0.53,0.93}
\definecolor{DodgerBlue3}{rgb}{0.09,0.45,0.80}
\definecolor{DodgerBlue4}{rgb}{0.06,0.31,0.55}
\definecolor{DodgerBlue}{rgb}{0.12,0.56,1.00}
\definecolor{FloralWhite}{rgb}{1.00,0.98,0.94}
\definecolor{ForestGreen}{rgb}{0.13,0.55,0.13}
\definecolor{GhostWhite}{rgb}{0.97,0.97,1.00}
\definecolor{GreenYellow}{rgb}{0.68,1.00,0.18}
\definecolor{HotPink1}{rgb}{1.00,0.43,0.71}
\definecolor{HotPink2}{rgb}{0.93,0.42,0.65}
\definecolor{HotPink3}{rgb}{0.80,0.38,0.56}
\definecolor{HotPink4}{rgb}{0.55,0.23,0.38}
\definecolor{HotPink}{rgb}{1.00,0.41,0.71}
\definecolor{IndianRed1}{rgb}{1.00,0.42,0.42}
\definecolor{IndianRed2}{rgb}{0.93,0.39,0.39}
\definecolor{IndianRed3}{rgb}{0.80,0.33,0.33}
\definecolor{IndianRed4}{rgb}{0.55,0.23,0.23}
\definecolor{IndianRed}{rgb}{0.80,0.36,0.36}
\definecolor{LavenderBlush1}{rgb}{1.00,0.94,0.96}
\definecolor{LavenderBlush2}{rgb}{0.93,0.88,0.90}
\definecolor{LavenderBlush3}{rgb}{0.80,0.76,0.77}
\definecolor{LavenderBlush4}{rgb}{0.55,0.51,0.53}
\definecolor{LavenderBlush}{rgb}{1.00,0.94,0.96}
\definecolor{LawnGreen}{rgb}{0.49,0.99,0.00}
\definecolor{LemonChiffon1}{rgb}{1.00,0.98,0.80}
\definecolor{LemonChiffon2}{rgb}{0.93,0.91,0.75}
\definecolor{LemonChiffon3}{rgb}{0.80,0.79,0.65}
\definecolor{LemonChiffon4}{rgb}{0.55,0.54,0.44}
\definecolor{LemonChiffon}{rgb}{1.00,0.98,0.80}
\definecolor{LightBlue1}{rgb}{0.75,0.94,1.00}
\definecolor{LightBlue2}{rgb}{0.70,0.87,0.93}
\definecolor{LightBlue3}{rgb}{0.60,0.75,0.80}
\definecolor{LightBlue4}{rgb}{0.41,0.51,0.55}
\definecolor{LightBlue}{rgb}{0.68,0.85,0.90}
\definecolor{LightCoral}{rgb}{0.94,0.50,0.50}
\definecolor{LightCyan1}{rgb}{0.88,1.00,1.00}
\definecolor{LightCyan2}{rgb}{0.82,0.93,0.93}
\definecolor{LightCyan3}{rgb}{0.71,0.80,0.80}
\definecolor{LightCyan4}{rgb}{0.48,0.55,0.55}
\definecolor{LightCyan}{rgb}{0.88,1.00,1.00}
\definecolor{LightGoldenrod1}{rgb}{1.00,0.93,0.55}
\definecolor{LightGoldenrod2}{rgb}{0.93,0.86,0.51}
\definecolor{LightGoldenrod3}{rgb}{0.80,0.75,0.44}
\definecolor{LightGoldenrod4}{rgb}{0.55,0.51,0.30}
\definecolor{LightGoldenrodYellow}{rgb}{0.98,0.98,0.82}
\definecolor{LightGoldenrod}{rgb}{0.93,0.87,0.51}
\definecolor{LightGray}{rgb}{0.83,0.83,0.83}
\definecolor{LightGreen}{rgb}{0.56,0.93,0.56}
\definecolor{LightGrey}{rgb}{0.83,0.83,0.83}
\definecolor{LightPink1}{rgb}{1.00,0.68,0.73}
\definecolor{LightPink2}{rgb}{0.93,0.64,0.68}
\definecolor{LightPink3}{rgb}{0.80,0.55,0.58}
\definecolor{LightPink4}{rgb}{0.55,0.37,0.40}
\definecolor{LightPink}{rgb}{1.00,0.71,0.76}
\definecolor{LightSalmon1}{rgb}{1.00,0.63,0.48}
\definecolor{LightSalmon2}{rgb}{0.93,0.58,0.45}
\definecolor{LightSalmon3}{rgb}{0.80,0.51,0.38}
\definecolor{LightSalmon4}{rgb}{0.55,0.34,0.26}
\definecolor{LightSalmon}{rgb}{1.00,0.63,0.48}
\definecolor{LightSeaGreen}{rgb}{0.13,0.70,0.67}
\definecolor{LightSkyBlue1}{rgb}{0.69,0.89,1.00}
\definecolor{LightSkyBlue2}{rgb}{0.64,0.83,0.93}
\definecolor{LightSkyBlue3}{rgb}{0.55,0.71,0.80}
\definecolor{LightSkyBlue4}{rgb}{0.38,0.48,0.55}
\definecolor{LightSkyBlue}{rgb}{0.53,0.81,0.98}
\definecolor{LightSlateBlue}{rgb}{0.52,0.44,1.00}
\definecolor{LightSlateGray}{rgb}{0.47,0.53,0.60}
\definecolor{LightSlateGrey}{rgb}{0.47,0.53,0.60}
\definecolor{LightSteelBlue1}{rgb}{0.79,0.88,1.00}
\definecolor{LightSteelBlue2}{rgb}{0.74,0.82,0.93}
\definecolor{LightSteelBlue3}{rgb}{0.64,0.71,0.80}
\definecolor{LightSteelBlue4}{rgb}{0.43,0.48,0.55}
\definecolor{LightSteelBlue}{rgb}{0.69,0.77,0.87}
\definecolor{LightYellow1}{rgb}{1.00,1.00,0.88}
\definecolor{LightYellow2}{rgb}{0.93,0.93,0.82}
\definecolor{LightYellow3}{rgb}{0.80,0.80,0.71}
\definecolor{LightYellow4}{rgb}{0.55,0.55,0.48}
\definecolor{LightYellow}{rgb}{1.00,1.00,0.88}
\definecolor{LimeGreen}{rgb}{0.20,0.80,0.20}
\definecolor{MediumAquamarine}{rgb}{0.40,0.80,0.67}
\definecolor{MediumBlue}{rgb}{0.00,0.00,0.80}
\definecolor{MediumOrchid1}{rgb}{0.88,0.40,1.00}
\definecolor{MediumOrchid2}{rgb}{0.82,0.37,0.93}
\definecolor{MediumOrchid3}{rgb}{0.71,0.32,0.80}
\definecolor{MediumOrchid4}{rgb}{0.48,0.22,0.55}
\definecolor{MediumOrchid}{rgb}{0.73,0.33,0.83}
\definecolor{MediumPurple1}{rgb}{0.67,0.51,1.00}
\definecolor{MediumPurple2}{rgb}{0.62,0.47,0.93}
\definecolor{MediumPurple3}{rgb}{0.54,0.41,0.80}
\definecolor{MediumPurple4}{rgb}{0.36,0.28,0.55}
\definecolor{MediumPurple}{rgb}{0.58,0.44,0.86}
\definecolor{MediumSeaGreen}{rgb}{0.24,0.70,0.44}
\definecolor{MediumSlateBlue}{rgb}{0.48,0.41,0.93}
\definecolor{MediumSpringGreen}{rgb}{0.00,0.98,0.60}
\definecolor{MediumTurquoise}{rgb}{0.28,0.82,0.80}
\definecolor{MediumVioletRed}{rgb}{0.78,0.08,0.52}
\definecolor{MidnightBlue}{rgb}{0.10,0.10,0.44}
\definecolor{MintCream}{rgb}{0.96,1.00,0.98}
\definecolor{MistyRose1}{rgb}{1.00,0.89,0.88}
\definecolor{MistyRose2}{rgb}{0.93,0.84,0.82}
\definecolor{MistyRose3}{rgb}{0.80,0.72,0.71}
\definecolor{MistyRose4}{rgb}{0.55,0.49,0.48}
\definecolor{MistyRose}{rgb}{1.00,0.89,0.88}
\definecolor{NavajoWhite1}{rgb}{1.00,0.87,0.68}
\definecolor{NavajoWhite2}{rgb}{0.93,0.81,0.63}
\definecolor{NavajoWhite3}{rgb}{0.80,0.70,0.55}
\definecolor{NavajoWhite4}{rgb}{0.55,0.47,0.37}
\definecolor{NavajoWhite}{rgb}{1.00,0.87,0.68}
\definecolor{NavyBlue}{rgb}{0.00,0.00,0.50}
\definecolor{OldLace}{rgb}{0.99,0.96,0.90}
\definecolor{OliveDrab1}{rgb}{0.75,1.00,0.24}
\definecolor{OliveDrab2}{rgb}{0.70,0.93,0.23}
\definecolor{OliveDrab3}{rgb}{0.60,0.80,0.20}
\definecolor{OliveDrab4}{rgb}{0.41,0.55,0.13}
\definecolor{OliveDrab}{rgb}{0.42,0.56,0.14}
\definecolor{OrangeRed1}{rgb}{1.00,0.27,0.00}
\definecolor{OrangeRed2}{rgb}{0.93,0.25,0.00}
\definecolor{OrangeRed3}{rgb}{0.80,0.22,0.00}
\definecolor{OrangeRed4}{rgb}{0.55,0.15,0.00}
\definecolor{OrangeRed}{rgb}{1.00,0.27,0.00}
\definecolor{PaleGoldenrod}{rgb}{0.93,0.91,0.67}
\definecolor{PaleGreen1}{rgb}{0.60,1.00,0.60}
\definecolor{PaleGreen2}{rgb}{0.56,0.93,0.56}
\definecolor{PaleGreen3}{rgb}{0.49,0.80,0.49}
\definecolor{PaleGreen4}{rgb}{0.33,0.55,0.33}
\definecolor{PaleGreen}{rgb}{0.60,0.98,0.60}
\definecolor{PaleTurquoise1}{rgb}{0.73,1.00,1.00}
\definecolor{PaleTurquoise2}{rgb}{0.68,0.93,0.93}
\definecolor{PaleTurquoise3}{rgb}{0.59,0.80,0.80}
\definecolor{PaleTurquoise4}{rgb}{0.40,0.55,0.55}
\definecolor{PaleTurquoise}{rgb}{0.69,0.93,0.93}
\definecolor{PaleVioletRed1}{rgb}{1.00,0.51,0.67}
\definecolor{PaleVioletRed2}{rgb}{0.93,0.47,0.62}
\definecolor{PaleVioletRed3}{rgb}{0.80,0.41,0.54}
\definecolor{PaleVioletRed4}{rgb}{0.55,0.28,0.36}
\definecolor{PaleVioletRed}{rgb}{0.86,0.44,0.58}
\definecolor{PapayaWhip}{rgb}{1.00,0.94,0.84}
\definecolor{PeachPuff1}{rgb}{1.00,0.85,0.73}
\definecolor{PeachPuff2}{rgb}{0.93,0.80,0.68}
\definecolor{PeachPuff3}{rgb}{0.80,0.69,0.58}
\definecolor{PeachPuff4}{rgb}{0.55,0.47,0.40}
\definecolor{PeachPuff}{rgb}{1.00,0.85,0.73}
\definecolor{PowderBlue}{rgb}{0.69,0.88,0.90}
\definecolor{RosyBrown1}{rgb}{1.00,0.76,0.76}
\definecolor{RosyBrown2}{rgb}{0.93,0.71,0.71}
\definecolor{RosyBrown3}{rgb}{0.80,0.61,0.61}
\definecolor{RosyBrown4}{rgb}{0.55,0.41,0.41}
\definecolor{RosyBrown}{rgb}{0.74,0.56,0.56}
\definecolor{RoyalBlue1}{rgb}{0.28,0.46,1.00}
\definecolor{RoyalBlue2}{rgb}{0.26,0.43,0.93}
\definecolor{RoyalBlue3}{rgb}{0.23,0.37,0.80}
\definecolor{RoyalBlue4}{rgb}{0.15,0.25,0.55}
\definecolor{RoyalBlue}{rgb}{0.25,0.41,0.88}
\definecolor{SaddleBrown}{rgb}{0.55,0.27,0.07}
\definecolor{SandyBrown}{rgb}{0.96,0.64,0.38}
\definecolor{SeaGreen1}{rgb}{0.33,1.00,0.62}
\definecolor{SeaGreen2}{rgb}{0.31,0.93,0.58}
\definecolor{SeaGreen3}{rgb}{0.26,0.80,0.50}
\definecolor{SeaGreen4}{rgb}{0.18,0.55,0.34}
\definecolor{SeaGreen}{rgb}{0.18,0.55,0.34}
\definecolor{SkyBlue1}{rgb}{0.53,0.81,1.00}
\definecolor{SkyBlue2}{rgb}{0.49,0.75,0.93}
\definecolor{SkyBlue3}{rgb}{0.42,0.65,0.80}
\definecolor{SkyBlue4}{rgb}{0.29,0.44,0.55}
\definecolor{SkyBlue}{rgb}{0.53,0.81,0.92}
\definecolor{SlateBlue1}{rgb}{0.51,0.44,1.00}
\definecolor{SlateBlue2}{rgb}{0.48,0.40,0.93}
\definecolor{SlateBlue3}{rgb}{0.41,0.35,0.80}
\definecolor{SlateBlue4}{rgb}{0.28,0.24,0.55}
\definecolor{SlateBlue}{rgb}{0.42,0.35,0.80}
\definecolor{SlateGray1}{rgb}{0.78,0.89,1.00}
\definecolor{SlateGray2}{rgb}{0.73,0.83,0.93}
\definecolor{SlateGray3}{rgb}{0.62,0.71,0.80}
\definecolor{SlateGray4}{rgb}{0.42,0.48,0.55}
\definecolor{SlateGray}{rgb}{0.44,0.50,0.56}
\definecolor{SlateGrey}{rgb}{0.44,0.50,0.56}
\definecolor{SpringGreen1}{rgb}{0.00,1.00,0.50}
\definecolor{SpringGreen2}{rgb}{0.00,0.93,0.46}
\definecolor{SpringGreen3}{rgb}{0.00,0.80,0.40}
\definecolor{SpringGreen4}{rgb}{0.00,0.55,0.27}
\definecolor{SpringGreen}{rgb}{0.00,1.00,0.50}
\definecolor{SteelBlue1}{rgb}{0.39,0.72,1.00}
\definecolor{SteelBlue2}{rgb}{0.36,0.67,0.93}
\definecolor{SteelBlue3}{rgb}{0.31,0.58,0.80}
\definecolor{SteelBlue4}{rgb}{0.21,0.39,0.55}
\definecolor{SteelBlue}{rgb}{0.27,0.51,0.71}
\definecolor{VioletRed1}{rgb}{1.00,0.24,0.59}
\definecolor{VioletRed2}{rgb}{0.93,0.23,0.55}
\definecolor{VioletRed3}{rgb}{0.80,0.20,0.47}
\definecolor{VioletRed4}{rgb}{0.55,0.13,0.32}
\definecolor{VioletRed}{rgb}{0.82,0.13,0.56}
\definecolor{WhiteSmoke}{rgb}{0.96,0.96,0.96}
\definecolor{YellowGreen}{rgb}{0.60,0.80,0.20}
\definecolor{aliceblue}{rgb}{0.94,0.97,1.00}
\definecolor{antiquewhite}{rgb}{0.98,0.92,0.84}
\definecolor{aquamarine1}{rgb}{0.50,1.00,0.83}
\definecolor{aquamarine2}{rgb}{0.46,0.93,0.78}
\definecolor{aquamarine3}{rgb}{0.40,0.80,0.67}
\definecolor{aquamarine4}{rgb}{0.27,0.55,0.45}
\definecolor{aquamarine}{rgb}{0.50,1.00,0.83}
\definecolor{azure1}{rgb}{0.94,1.00,1.00}
\definecolor{azure2}{rgb}{0.88,0.93,0.93}
\definecolor{azure3}{rgb}{0.76,0.80,0.80}
\definecolor{azure4}{rgb}{0.51,0.55,0.55}
\definecolor{azure}{rgb}{0.94,1.00,1.00}
\definecolor{beige}{rgb}{0.96,0.96,0.86}
\definecolor{bisque1}{rgb}{1.00,0.89,0.77}
\definecolor{bisque2}{rgb}{0.93,0.84,0.72}
\definecolor{bisque3}{rgb}{0.80,0.72,0.62}
\definecolor{bisque4}{rgb}{0.55,0.49,0.42}
\definecolor{bisque}{rgb}{1.00,0.89,0.77}
\definecolor{black}{rgb}{0.00,0.00,0.00}
\definecolor{blanchedalmond}{rgb}{1.00,0.92,0.80}
\definecolor{blue1}{rgb}{0.00,0.00,1.00}
\definecolor{blue2}{rgb}{0.00,0.00,0.93}
\definecolor{blue3}{rgb}{0.00,0.00,0.80}
\definecolor{blue4}{rgb}{0.00,0.00,0.55}
\definecolor{blueviolet}{rgb}{0.54,0.17,0.89}
\definecolor{blue}{rgb}{0.00,0.00,1.00}
\definecolor{brown1}{rgb}{1.00,0.25,0.25}
\definecolor{brown2}{rgb}{0.93,0.23,0.23}
\definecolor{brown3}{rgb}{0.80,0.20,0.20}
\definecolor{brown4}{rgb}{0.55,0.14,0.14}
\definecolor{brown}{rgb}{0.65,0.16,0.16}
\definecolor{burlywood1}{rgb}{1.00,0.83,0.61}
\definecolor{burlywood2}{rgb}{0.93,0.77,0.57}
\definecolor{burlywood3}{rgb}{0.80,0.67,0.49}
\definecolor{burlywood4}{rgb}{0.55,0.45,0.33}
\definecolor{burlywood}{rgb}{0.87,0.72,0.53}
\definecolor{cadetblue}{rgb}{0.37,0.62,0.63}
\definecolor{chartreuse1}{rgb}{0.50,1.00,0.00}
\definecolor{chartreuse2}{rgb}{0.46,0.93,0.00}
\definecolor{chartreuse3}{rgb}{0.40,0.80,0.00}
\definecolor{chartreuse4}{rgb}{0.27,0.55,0.00}
\definecolor{chartreuse}{rgb}{0.50,1.00,0.00}
\definecolor{chocolate1}{rgb}{1.00,0.50,0.14}
\definecolor{chocolate2}{rgb}{0.93,0.46,0.13}
\definecolor{chocolate3}{rgb}{0.80,0.40,0.11}
\definecolor{chocolate4}{rgb}{0.55,0.27,0.07}
\definecolor{chocolate}{rgb}{0.82,0.41,0.12}
\definecolor{coral1}{rgb}{1.00,0.45,0.34}
\definecolor{coral2}{rgb}{0.93,0.42,0.31}
\definecolor{coral3}{rgb}{0.80,0.36,0.27}
\definecolor{coral4}{rgb}{0.55,0.24,0.18}
\definecolor{coral}{rgb}{1.00,0.50,0.31}
\definecolor{cornflowerblue}{rgb}{0.39,0.58,0.93}
\definecolor{cornsilk1}{rgb}{1.00,0.97,0.86}
\definecolor{cornsilk2}{rgb}{0.93,0.91,0.80}
\definecolor{cornsilk3}{rgb}{0.80,0.78,0.69}
\definecolor{cornsilk4}{rgb}{0.55,0.53,0.47}
\definecolor{cornsilk}{rgb}{1.00,0.97,0.86}
\definecolor{cyan1}{rgb}{0.00,1.00,1.00}
\definecolor{cyan2}{rgb}{0.00,0.93,0.93}
\definecolor{cyan3}{rgb}{0.00,0.80,0.80}
\definecolor{cyan4}{rgb}{0.00,0.55,0.55}
\definecolor{cyan}{rgb}{0.00,1.00,1.00}
\definecolor{darkblue}{rgb}{0.00,0.00,0.55}
\definecolor{darkcyan}{rgb}{0.00,0.55,0.55}
\definecolor{darkgoldenrod}{rgb}{0.72,0.53,0.04}
\definecolor{darkgray}{rgb}{0.66,0.66,0.66}
\definecolor{darkgreen}{rgb}{0.00,0.39,0.00}
\definecolor{darkgrey}{rgb}{0.66,0.66,0.66}
\definecolor{darkkhaki}{rgb}{0.74,0.72,0.42}
\definecolor{darkmagenta}{rgb}{0.55,0.00,0.55}
\definecolor{darkolive}{rgb}{0.33,0.42,0.18}
\definecolor{darkorange}{rgb}{1.00,0.55,0.00}
\definecolor{darkorchid}{rgb}{0.60,0.20,0.80}
\definecolor{darkred}{rgb}{0.55,0.00,0.00}
\definecolor{darksalmon}{rgb}{0.91,0.59,0.48}
\definecolor{darksea}{rgb}{0.56,0.74,0.56}
\definecolor{darkslate}{rgb}{0.18,0.31,0.31}
\definecolor{darkslate}{rgb}{0.18,0.31,0.31}
\definecolor{darkslate}{rgb}{0.28,0.24,0.55}
\definecolor{darkturquoise}{rgb}{0.00,0.81,0.82}
\definecolor{darkviolet}{rgb}{0.58,0.00,0.83}
\definecolor{deeppink}{rgb}{1.00,0.08,0.58}
\definecolor{deepsky}{rgb}{0.00,0.75,1.00}
\definecolor{dimgray}{rgb}{0.41,0.41,0.41}
\definecolor{dimgrey}{rgb}{0.41,0.41,0.41}
\definecolor{dodgerblue}{rgb}{0.12,0.56,1.00}
\definecolor{firebrick1}{rgb}{1.00,0.19,0.19}
\definecolor{firebrick2}{rgb}{0.93,0.17,0.17}
\definecolor{firebrick3}{rgb}{0.80,0.15,0.15}
\definecolor{firebrick4}{rgb}{0.55,0.10,0.10}
\definecolor{firebrick}{rgb}{0.70,0.13,0.13}
\definecolor{floralwhite}{rgb}{1.00,0.98,0.94}
\definecolor{forestgreen}{rgb}{0.13,0.55,0.13}
\definecolor{gainsboro}{rgb}{0.86,0.86,0.86}
\definecolor{ghostwhite}{rgb}{0.97,0.97,1.00}
\definecolor{gold1}{rgb}{1.00,0.84,0.00}
\definecolor{gold2}{rgb}{0.93,0.79,0.00}
\definecolor{gold3}{rgb}{0.80,0.68,0.00}
\definecolor{gold4}{rgb}{0.55,0.46,0.00}
\definecolor{goldenrod1}{rgb}{1.00,0.76,0.15}
\definecolor{goldenrod2}{rgb}{0.93,0.71,0.13}
\definecolor{goldenrod3}{rgb}{0.80,0.61,0.11}
\definecolor{goldenrod4}{rgb}{0.55,0.41,0.08}
\definecolor{goldenrod}{rgb}{0.85,0.65,0.13}
\definecolor{gold}{rgb}{1.00,0.84,0.00}
\definecolor{gray0}{rgb}{0.00,0.00,0.00}
\definecolor{gray100}{rgb}{1.00,1.00,1.00}
\definecolor{gray10}{rgb}{0.10,0.10,0.10}
\definecolor{gray11}{rgb}{0.11,0.11,0.11}
\definecolor{gray12}{rgb}{0.12,0.12,0.12}
\definecolor{gray13}{rgb}{0.13,0.13,0.13}
\definecolor{gray14}{rgb}{0.14,0.14,0.14}
\definecolor{gray15}{rgb}{0.15,0.15,0.15}
\definecolor{gray16}{rgb}{0.16,0.16,0.16}
\definecolor{gray17}{rgb}{0.17,0.17,0.17}
\definecolor{gray18}{rgb}{0.18,0.18,0.18}
\definecolor{gray19}{rgb}{0.19,0.19,0.19}
\definecolor{gray1}{rgb}{0.01,0.01,0.01}
\definecolor{gray20}{rgb}{0.20,0.20,0.20}
\definecolor{gray21}{rgb}{0.21,0.21,0.21}
\definecolor{gray22}{rgb}{0.22,0.22,0.22}
\definecolor{gray23}{rgb}{0.23,0.23,0.23}
\definecolor{gray24}{rgb}{0.24,0.24,0.24}
\definecolor{gray25}{rgb}{0.25,0.25,0.25}
\definecolor{gray26}{rgb}{0.26,0.26,0.26}
\definecolor{gray27}{rgb}{0.27,0.27,0.27}
\definecolor{gray28}{rgb}{0.28,0.28,0.28}
\definecolor{gray29}{rgb}{0.29,0.29,0.29}
\definecolor{gray2}{rgb}{0.02,0.02,0.02}
\definecolor{gray30}{rgb}{0.30,0.30,0.30}
\definecolor{gray31}{rgb}{0.31,0.31,0.31}
\definecolor{gray32}{rgb}{0.32,0.32,0.32}
\definecolor{gray33}{rgb}{0.33,0.33,0.33}
\definecolor{gray34}{rgb}{0.34,0.34,0.34}
\definecolor{gray35}{rgb}{0.35,0.35,0.35}
\definecolor{gray36}{rgb}{0.36,0.36,0.36}
\definecolor{gray37}{rgb}{0.37,0.37,0.37}
\definecolor{gray38}{rgb}{0.38,0.38,0.38}
\definecolor{gray39}{rgb}{0.39,0.39,0.39}
\definecolor{gray3}{rgb}{0.03,0.03,0.03}
\definecolor{gray40}{rgb}{0.40,0.40,0.40}
\definecolor{gray41}{rgb}{0.41,0.41,0.41}
\definecolor{gray42}{rgb}{0.42,0.42,0.42}
\definecolor{gray43}{rgb}{0.43,0.43,0.43}
\definecolor{gray44}{rgb}{0.44,0.44,0.44}
\definecolor{gray45}{rgb}{0.45,0.45,0.45}
\definecolor{gray46}{rgb}{0.46,0.46,0.46}
\definecolor{gray47}{rgb}{0.47,0.47,0.47}
\definecolor{gray48}{rgb}{0.48,0.48,0.48}
\definecolor{gray49}{rgb}{0.49,0.49,0.49}
\definecolor{gray4}{rgb}{0.04,0.04,0.04}
\definecolor{gray50}{rgb}{0.50,0.50,0.50}
\definecolor{gray51}{rgb}{0.51,0.51,0.51}
\definecolor{gray52}{rgb}{0.52,0.52,0.52}
\definecolor{gray53}{rgb}{0.53,0.53,0.53}
\definecolor{gray54}{rgb}{0.54,0.54,0.54}
\definecolor{gray55}{rgb}{0.55,0.55,0.55}
\definecolor{gray56}{rgb}{0.56,0.56,0.56}
\definecolor{gray57}{rgb}{0.57,0.57,0.57}
\definecolor{gray58}{rgb}{0.58,0.58,0.58}
\definecolor{gray59}{rgb}{0.59,0.59,0.59}
\definecolor{gray5}{rgb}{0.05,0.05,0.05}
\definecolor{gray60}{rgb}{0.60,0.60,0.60}
\definecolor{gray61}{rgb}{0.61,0.61,0.61}
\definecolor{gray62}{rgb}{0.62,0.62,0.62}
\definecolor{gray63}{rgb}{0.63,0.63,0.63}
\definecolor{gray64}{rgb}{0.64,0.64,0.64}
\definecolor{gray65}{rgb}{0.65,0.65,0.65}
\definecolor{gray66}{rgb}{0.66,0.66,0.66}
\definecolor{gray67}{rgb}{0.67,0.67,0.67}
\definecolor{gray68}{rgb}{0.68,0.68,0.68}
\definecolor{gray69}{rgb}{0.69,0.69,0.69}
\definecolor{gray6}{rgb}{0.06,0.06,0.06}
\definecolor{gray70}{rgb}{0.70,0.70,0.70}
\definecolor{gray71}{rgb}{0.71,0.71,0.71}
\definecolor{gray72}{rgb}{0.72,0.72,0.72}
\definecolor{gray73}{rgb}{0.73,0.73,0.73}
\definecolor{gray74}{rgb}{0.74,0.74,0.74}
\definecolor{gray75}{rgb}{0.75,0.75,0.75}
\definecolor{gray76}{rgb}{0.76,0.76,0.76}
\definecolor{gray77}{rgb}{0.77,0.77,0.77}
\definecolor{gray78}{rgb}{0.78,0.78,0.78}
\definecolor{gray79}{rgb}{0.79,0.79,0.79}
\definecolor{gray7}{rgb}{0.07,0.07,0.07}
\definecolor{gray80}{rgb}{0.80,0.80,0.80}
\definecolor{gray81}{rgb}{0.81,0.81,0.81}
\definecolor{gray82}{rgb}{0.82,0.82,0.82}
\definecolor{gray83}{rgb}{0.83,0.83,0.83}
\definecolor{gray84}{rgb}{0.84,0.84,0.84}
\definecolor{gray85}{rgb}{0.85,0.85,0.85}
\definecolor{gray86}{rgb}{0.86,0.86,0.86}
\definecolor{gray87}{rgb}{0.87,0.87,0.87}
\definecolor{gray88}{rgb}{0.88,0.88,0.88}
\definecolor{gray89}{rgb}{0.89,0.89,0.89}
\definecolor{gray8}{rgb}{0.08,0.08,0.08}
\definecolor{gray90}{rgb}{0.90,0.90,0.90}
\definecolor{gray91}{rgb}{0.91,0.91,0.91}
\definecolor{gray92}{rgb}{0.92,0.92,0.92}
\definecolor{gray93}{rgb}{0.93,0.93,0.93}
\definecolor{gray94}{rgb}{0.94,0.94,0.94}
\definecolor{gray95}{rgb}{0.95,0.95,0.95}
\definecolor{gray96}{rgb}{0.96,0.96,0.96}
\definecolor{gray97}{rgb}{0.97,0.97,0.97}
\definecolor{gray98}{rgb}{0.98,0.98,0.98}
\definecolor{gray99}{rgb}{0.99,0.99,0.99}
\definecolor{gray9}{rgb}{0.09,0.09,0.09}
\definecolor{gray}{rgb}{0.75,0.75,0.75}
\definecolor{green1}{rgb}{0.00,1.00,0.00}
\definecolor{green2}{rgb}{0.00,0.93,0.00}
\definecolor{green3}{rgb}{0.00,0.80,0.00}
\definecolor{green4}{rgb}{0.00,0.55,0.00}
\definecolor{greenyellow}{rgb}{0.68,1.00,0.18}
\definecolor{green}{rgb}{0.00,1.00,0.00}
\definecolor{grey0}{rgb}{0.00,0.00,0.00}
\definecolor{grey100}{rgb}{1.00,1.00,1.00}
\definecolor{grey10}{rgb}{0.10,0.10,0.10}
\definecolor{grey11}{rgb}{0.11,0.11,0.11}
\definecolor{grey12}{rgb}{0.12,0.12,0.12}
\definecolor{grey13}{rgb}{0.13,0.13,0.13}
\definecolor{grey14}{rgb}{0.14,0.14,0.14}
\definecolor{grey15}{rgb}{0.15,0.15,0.15}
\definecolor{grey16}{rgb}{0.16,0.16,0.16}
\definecolor{grey17}{rgb}{0.17,0.17,0.17}
\definecolor{grey18}{rgb}{0.18,0.18,0.18}
\definecolor{grey19}{rgb}{0.19,0.19,0.19}
\definecolor{grey1}{rgb}{0.01,0.01,0.01}
\definecolor{grey20}{rgb}{0.20,0.20,0.20}
\definecolor{grey21}{rgb}{0.21,0.21,0.21}
\definecolor{grey22}{rgb}{0.22,0.22,0.22}
\definecolor{grey23}{rgb}{0.23,0.23,0.23}
\definecolor{grey24}{rgb}{0.24,0.24,0.24}
\definecolor{grey25}{rgb}{0.25,0.25,0.25}
\definecolor{grey26}{rgb}{0.26,0.26,0.26}
\definecolor{grey27}{rgb}{0.27,0.27,0.27}
\definecolor{grey28}{rgb}{0.28,0.28,0.28}
\definecolor{grey29}{rgb}{0.29,0.29,0.29}
\definecolor{grey2}{rgb}{0.02,0.02,0.02}
\definecolor{grey30}{rgb}{0.30,0.30,0.30}
\definecolor{grey31}{rgb}{0.31,0.31,0.31}
\definecolor{grey32}{rgb}{0.32,0.32,0.32}
\definecolor{grey33}{rgb}{0.33,0.33,0.33}
\definecolor{grey34}{rgb}{0.34,0.34,0.34}
\definecolor{grey35}{rgb}{0.35,0.35,0.35}
\definecolor{grey36}{rgb}{0.36,0.36,0.36}
\definecolor{grey37}{rgb}{0.37,0.37,0.37}
\definecolor{grey38}{rgb}{0.38,0.38,0.38}
\definecolor{grey39}{rgb}{0.39,0.39,0.39}
\definecolor{grey3}{rgb}{0.03,0.03,0.03}
\definecolor{grey40}{rgb}{0.40,0.40,0.40}
\definecolor{grey41}{rgb}{0.41,0.41,0.41}
\definecolor{grey42}{rgb}{0.42,0.42,0.42}
\definecolor{grey43}{rgb}{0.43,0.43,0.43}
\definecolor{grey44}{rgb}{0.44,0.44,0.44}
\definecolor{grey45}{rgb}{0.45,0.45,0.45}
\definecolor{grey46}{rgb}{0.46,0.46,0.46}
\definecolor{grey47}{rgb}{0.47,0.47,0.47}
\definecolor{grey48}{rgb}{0.48,0.48,0.48}
\definecolor{grey49}{rgb}{0.49,0.49,0.49}
\definecolor{grey4}{rgb}{0.04,0.04,0.04}
\definecolor{grey50}{rgb}{0.50,0.50,0.50}
\definecolor{grey51}{rgb}{0.51,0.51,0.51}
\definecolor{grey52}{rgb}{0.52,0.52,0.52}
\definecolor{grey53}{rgb}{0.53,0.53,0.53}
\definecolor{grey54}{rgb}{0.54,0.54,0.54}
\definecolor{grey55}{rgb}{0.55,0.55,0.55}
\definecolor{grey56}{rgb}{0.56,0.56,0.56}
\definecolor{grey57}{rgb}{0.57,0.57,0.57}
\definecolor{grey58}{rgb}{0.58,0.58,0.58}
\definecolor{grey59}{rgb}{0.59,0.59,0.59}
\definecolor{grey5}{rgb}{0.05,0.05,0.05}
\definecolor{grey60}{rgb}{0.60,0.60,0.60}
\definecolor{grey61}{rgb}{0.61,0.61,0.61}
\definecolor{grey62}{rgb}{0.62,0.62,0.62}
\definecolor{grey63}{rgb}{0.63,0.63,0.63}
\definecolor{grey64}{rgb}{0.64,0.64,0.64}
\definecolor{grey65}{rgb}{0.65,0.65,0.65}
\definecolor{grey66}{rgb}{0.66,0.66,0.66}
\definecolor{grey67}{rgb}{0.67,0.67,0.67}
\definecolor{grey68}{rgb}{0.68,0.68,0.68}
\definecolor{grey69}{rgb}{0.69,0.69,0.69}
\definecolor{grey6}{rgb}{0.06,0.06,0.06}
\definecolor{grey70}{rgb}{0.70,0.70,0.70}
\definecolor{grey71}{rgb}{0.71,0.71,0.71}
\definecolor{grey72}{rgb}{0.72,0.72,0.72}
\definecolor{grey73}{rgb}{0.73,0.73,0.73}
\definecolor{grey74}{rgb}{0.74,0.74,0.74}
\definecolor{grey75}{rgb}{0.75,0.75,0.75}
\definecolor{grey76}{rgb}{0.76,0.76,0.76}
\definecolor{grey77}{rgb}{0.77,0.77,0.77}
\definecolor{grey78}{rgb}{0.78,0.78,0.78}
\definecolor{grey79}{rgb}{0.79,0.79,0.79}
\definecolor{grey7}{rgb}{0.07,0.07,0.07}
\definecolor{grey80}{rgb}{0.80,0.80,0.80}
\definecolor{grey81}{rgb}{0.81,0.81,0.81}
\definecolor{grey82}{rgb}{0.82,0.82,0.82}
\definecolor{grey83}{rgb}{0.83,0.83,0.83}
\definecolor{grey84}{rgb}{0.84,0.84,0.84}
\definecolor{grey85}{rgb}{0.85,0.85,0.85}
\definecolor{grey86}{rgb}{0.86,0.86,0.86}
\definecolor{grey87}{rgb}{0.87,0.87,0.87}
\definecolor{grey88}{rgb}{0.88,0.88,0.88}
\definecolor{grey89}{rgb}{0.89,0.89,0.89}
\definecolor{grey8}{rgb}{0.08,0.08,0.08}
\definecolor{grey90}{rgb}{0.90,0.90,0.90}
\definecolor{grey91}{rgb}{0.91,0.91,0.91}
\definecolor{grey92}{rgb}{0.92,0.92,0.92}
\definecolor{grey93}{rgb}{0.93,0.93,0.93}
\definecolor{grey94}{rgb}{0.94,0.94,0.94}
\definecolor{grey95}{rgb}{0.95,0.95,0.95}
\definecolor{grey96}{rgb}{0.96,0.96,0.96}
\definecolor{grey97}{rgb}{0.97,0.97,0.97}
\definecolor{grey98}{rgb}{0.98,0.98,0.98}
\definecolor{grey99}{rgb}{0.99,0.99,0.99}
\definecolor{grey9}{rgb}{0.09,0.09,0.09}
\definecolor{grey}{rgb}{0.75,0.75,0.75}
\definecolor{honeydew1}{rgb}{0.94,1.00,0.94}
\definecolor{honeydew2}{rgb}{0.88,0.93,0.88}
\definecolor{honeydew3}{rgb}{0.76,0.80,0.76}
\definecolor{honeydew4}{rgb}{0.51,0.55,0.51}
\definecolor{honeydew}{rgb}{0.94,1.00,0.94}
\definecolor{hotpink}{rgb}{1.00,0.41,0.71}
\definecolor{indianred}{rgb}{0.80,0.36,0.36}
\definecolor{ivory1}{rgb}{1.00,1.00,0.94}
\definecolor{ivory2}{rgb}{0.93,0.93,0.88}
\definecolor{ivory3}{rgb}{0.80,0.80,0.76}
\definecolor{ivory4}{rgb}{0.55,0.55,0.51}
\definecolor{ivory}{rgb}{1.00,1.00,0.94}
\definecolor{khaki1}{rgb}{1.00,0.96,0.56}
\definecolor{khaki2}{rgb}{0.93,0.90,0.52}
\definecolor{khaki3}{rgb}{0.80,0.78,0.45}
\definecolor{khaki4}{rgb}{0.55,0.53,0.31}
\definecolor{khaki}{rgb}{0.94,0.90,0.55}
\definecolor{lavenderblush}{rgb}{1.00,0.94,0.96}
\definecolor{lavender}{rgb}{0.90,0.90,0.98}
\definecolor{lawngreen}{rgb}{0.49,0.99,0.00}
\definecolor{lemonchiffon}{rgb}{1.00,0.98,0.80}
\definecolor{lightblue}{rgb}{0.68,0.85,0.90}
\definecolor{lightcoral}{rgb}{0.94,0.50,0.50}
\definecolor{lightcyan}{rgb}{0.88,1.00,1.00}
\definecolor{lightgoldenrod}{rgb}{0.93,0.87,0.51}
\definecolor{lightgoldenrod}{rgb}{0.98,0.98,0.82}
\definecolor{lightgray}{rgb}{0.83,0.83,0.83}
\definecolor{lightgreen}{rgb}{0.56,0.93,0.56}
\definecolor{lightgrey}{rgb}{0.83,0.83,0.83}
\definecolor{lightpink}{rgb}{1.00,0.71,0.76}
\definecolor{lightsalmon}{rgb}{1.00,0.63,0.48}
\definecolor{lightsea}{rgb}{0.13,0.70,0.67}
\definecolor{lightsky}{rgb}{0.53,0.81,0.98}
\definecolor{lightslate}{rgb}{0.47,0.53,0.60}
\definecolor{lightslate}{rgb}{0.47,0.53,0.60}
\definecolor{lightslate}{rgb}{0.52,0.44,1.00}
\definecolor{lightsteel}{rgb}{0.69,0.77,0.87}
\definecolor{lightyellow}{rgb}{1.00,1.00,0.88}
\definecolor{limegreen}{rgb}{0.20,0.80,0.20}
\definecolor{linen}{rgb}{0.98,0.94,0.90}
\definecolor{magenta1}{rgb}{1.00,0.00,1.00}
\definecolor{magenta2}{rgb}{0.93,0.00,0.93}
\definecolor{magenta3}{rgb}{0.80,0.00,0.80}
\definecolor{magenta4}{rgb}{0.55,0.00,0.55}
\definecolor{magenta}{rgb}{1.00,0.00,1.00}
\definecolor{maroon1}{rgb}{1.00,0.20,0.70}
\definecolor{maroon2}{rgb}{0.93,0.19,0.65}
\definecolor{maroon3}{rgb}{0.80,0.16,0.56}
\definecolor{maroon4}{rgb}{0.55,0.11,0.38}
\definecolor{maroon}{rgb}{0.69,0.19,0.38}
\definecolor{mediumaquamarine}{rgb}{0.40,0.80,0.67}
\definecolor{mediumblue}{rgb}{0.00,0.00,0.80}
\definecolor{mediumorchid}{rgb}{0.73,0.33,0.83}
\definecolor{mediumpurple}{rgb}{0.58,0.44,0.86}
\definecolor{mediumsea}{rgb}{0.24,0.70,0.44}
\definecolor{mediumslate}{rgb}{0.48,0.41,0.93}
\definecolor{mediumspring}{rgb}{0.00,0.98,0.60}
\definecolor{mediumturquoise}{rgb}{0.28,0.82,0.80}
\definecolor{mediumviolet}{rgb}{0.78,0.08,0.52}
\definecolor{midnightblue}{rgb}{0.10,0.10,0.44}
\definecolor{mintcream}{rgb}{0.96,1.00,0.98}
\definecolor{mistyrose}{rgb}{1.00,0.89,0.88}
\definecolor{moccasin}{rgb}{1.00,0.89,0.71}
\definecolor{navajowhite}{rgb}{1.00,0.87,0.68}
\definecolor{navyblue}{rgb}{0.00,0.00,0.50}
\definecolor{navy}{rgb}{0.00,0.00,0.50}
\definecolor{oldlace}{rgb}{0.99,0.96,0.90}
\definecolor{olivedrab}{rgb}{0.42,0.56,0.14}
\definecolor{orange1}{rgb}{1.00,0.65,0.00}
\definecolor{orange2}{rgb}{0.93,0.60,0.00}
\definecolor{orange3}{rgb}{0.80,0.52,0.00}
\definecolor{orange4}{rgb}{0.55,0.35,0.00}
\definecolor{orangered}{rgb}{1.00,0.27,0.00}
\definecolor{orange}{rgb}{1.00,0.65,0.00}
\definecolor{orchid1}{rgb}{1.00,0.51,0.98}
\definecolor{orchid2}{rgb}{0.93,0.48,0.91}
\definecolor{orchid3}{rgb}{0.80,0.41,0.79}
\definecolor{orchid4}{rgb}{0.55,0.28,0.54}
\definecolor{orchid}{rgb}{0.85,0.44,0.84}
\definecolor{palegoldenrod}{rgb}{0.93,0.91,0.67}
\definecolor{palegreen}{rgb}{0.60,0.98,0.60}
\definecolor{paleturquoise}{rgb}{0.69,0.93,0.93}
\definecolor{paleviolet}{rgb}{0.86,0.44,0.58}
\definecolor{papayawhip}{rgb}{1.00,0.94,0.84}
\definecolor{peachpuff}{rgb}{1.00,0.85,0.73}
\definecolor{peru}{rgb}{0.80,0.52,0.25}
\definecolor{pink1}{rgb}{1.00,0.71,0.77}
\definecolor{pink2}{rgb}{0.93,0.66,0.72}
\definecolor{pink3}{rgb}{0.80,0.57,0.62}
\definecolor{pink4}{rgb}{0.55,0.39,0.42}
\definecolor{pink}{rgb}{1.00,0.75,0.80}
\definecolor{plum1}{rgb}{1.00,0.73,1.00}
\definecolor{plum2}{rgb}{0.93,0.68,0.93}
\definecolor{plum3}{rgb}{0.80,0.59,0.80}
\definecolor{plum4}{rgb}{0.55,0.40,0.55}
\definecolor{plum}{rgb}{0.87,0.63,0.87}
\definecolor{powderblue}{rgb}{0.69,0.88,0.90}
\definecolor{purple1}{rgb}{0.61,0.19,1.00}
\definecolor{purple2}{rgb}{0.57,0.17,0.93}
\definecolor{purple3}{rgb}{0.49,0.15,0.80}
\definecolor{purple4}{rgb}{0.33,0.10,0.55}
\definecolor{purple}{rgb}{0.63,0.13,0.94}
\definecolor{red1}{rgb}{1.00,0.00,0.00}
\definecolor{red2}{rgb}{0.93,0.00,0.00}
\definecolor{red3}{rgb}{0.80,0.00,0.00}
\definecolor{red4}{rgb}{0.55,0.00,0.00}
\definecolor{red}{rgb}{1.00,0.00,0.00}
\definecolor{rosybrown}{rgb}{0.74,0.56,0.56}
\definecolor{royalblue}{rgb}{0.25,0.41,0.88}
\definecolor{saddlebrown}{rgb}{0.55,0.27,0.07}
\definecolor{salmon1}{rgb}{1.00,0.55,0.41}
\definecolor{salmon2}{rgb}{0.93,0.51,0.38}
\definecolor{salmon3}{rgb}{0.80,0.44,0.33}
\definecolor{salmon4}{rgb}{0.55,0.30,0.22}
\definecolor{salmon}{rgb}{0.98,0.50,0.45}
\definecolor{sandybrown}{rgb}{0.96,0.64,0.38}
\definecolor{seagreen}{rgb}{0.18,0.55,0.34}
\definecolor{seashell1}{rgb}{1.00,0.96,0.93}
\definecolor{seashell2}{rgb}{0.93,0.90,0.87}
\definecolor{seashell3}{rgb}{0.80,0.77,0.75}
\definecolor{seashell4}{rgb}{0.55,0.53,0.51}
\definecolor{seashell}{rgb}{1.00,0.96,0.93}
\definecolor{sienna1}{rgb}{1.00,0.51,0.28}
\definecolor{sienna2}{rgb}{0.93,0.47,0.26}
\definecolor{sienna3}{rgb}{0.80,0.41,0.22}
\definecolor{sienna4}{rgb}{0.55,0.28,0.15}
\definecolor{sienna}{rgb}{0.63,0.32,0.18}
\definecolor{skyblue}{rgb}{0.53,0.81,0.92}
\definecolor{slateblue}{rgb}{0.42,0.35,0.80}
\definecolor{slategray}{rgb}{0.44,0.50,0.56}
\definecolor{slategrey}{rgb}{0.44,0.50,0.56}
\definecolor{snow1}{rgb}{1.00,0.98,0.98}
\definecolor{snow2}{rgb}{0.93,0.91,0.91}
\definecolor{snow3}{rgb}{0.80,0.79,0.79}
\definecolor{snow4}{rgb}{0.55,0.54,0.54}
\definecolor{snow}{rgb}{1.00,0.98,0.98}
\definecolor{springgreen}{rgb}{0.00,1.00,0.50}
\definecolor{steelblue}{rgb}{0.27,0.51,0.71}
\definecolor{tan1}{rgb}{1.00,0.65,0.31}
\definecolor{tan2}{rgb}{0.93,0.60,0.29}
\definecolor{tan3}{rgb}{0.80,0.52,0.25}
\definecolor{tan4}{rgb}{0.55,0.35,0.17}
\definecolor{tan}{rgb}{0.82,0.71,0.55}
\definecolor{thistle1}{rgb}{1.00,0.88,1.00}
\definecolor{thistle2}{rgb}{0.93,0.82,0.93}
\definecolor{thistle3}{rgb}{0.80,0.71,0.80}
\definecolor{thistle4}{rgb}{0.55,0.48,0.55}
\definecolor{thistle}{rgb}{0.85,0.75,0.85}
\definecolor{tomato1}{rgb}{1.00,0.39,0.28}
\definecolor{tomato2}{rgb}{0.93,0.36,0.26}
\definecolor{tomato3}{rgb}{0.80,0.31,0.22}
\definecolor{tomato4}{rgb}{0.55,0.21,0.15}
\definecolor{tomato}{rgb}{1.00,0.39,0.28}
\definecolor{turquoise1}{rgb}{0.00,0.96,1.00}
\definecolor{turquoise2}{rgb}{0.00,0.90,0.93}
\definecolor{turquoise3}{rgb}{0.00,0.77,0.80}
\definecolor{turquoise4}{rgb}{0.00,0.53,0.55}
\definecolor{turquoise}{rgb}{0.25,0.88,0.82}
\definecolor{violetred}{rgb}{0.82,0.13,0.56}
\definecolor{violet}{rgb}{0.93,0.51,0.93}
\definecolor{wheat1}{rgb}{1.00,0.91,0.73}
\definecolor{wheat2}{rgb}{0.93,0.85,0.68}
\definecolor{wheat3}{rgb}{0.80,0.73,0.59}
\definecolor{wheat4}{rgb}{0.55,0.49,0.40}
\definecolor{wheat}{rgb}{0.96,0.87,0.70}
\definecolor{whitesmoke}{rgb}{0.96,0.96,0.96}
\definecolor{white}{rgb}{1.00,1.00,1.00}
\definecolor{yellow1}{rgb}{1.00,1.00,0.00}
\definecolor{yellow2}{rgb}{0.93,0.93,0.00}
\definecolor{yellow3}{rgb}{0.80,0.80,0.00}
\definecolor{yellow4}{rgb}{0.55,0.55,0.00}
\definecolor{yellowgreen}{rgb}{0.60,0.80,0.20}
\definecolor{yellow}{rgb}{1.00,1.00,0.00}

\usepackage[colorinlistoftodos]{todonotes}

\numberwithin{equation}{section}
\newtheorem{thm}{Theorem}[section]
\newtheorem{cor}[thm]{Corollary}
\newtheorem{lem}[thm]{Lemma}
\newtheorem{prop}[thm]{Proposition}

\theoremstyle{definition}
\newtheorem{defn}[thm]{Definition}

\newcommand{\lxr}{\longrightarrow}

\newcommand{\aGP}{\mathscr GP}

\newsavebox{\proofbox}
\savebox{\proofbox}{\begin{picture}(7,7)%
  \put(0,0){\framebox(7,7){}}\end{picture}}

\begin{document}

\title[]{From Gorenstein derived equivalences to stable functors of Gorenstein projective modules}
\author[]{Nan Gao$^{*}$, Chi-Heng Zhang, Jing Ma
}
\address{Department of Mathematics, Shanghai University, Shanghai 200444, PR China}
\thanks{* is the corresponding author.}
\email{nangao@shu.edu.cn, zchmath@shu.edu.cn, majingmmmm@shu.edu.cn  }
\thanks{Supported by the National Natural Science Foundation of China (Grant No.11771272 and 11871326).}

\date{\today}



\begin{abstract}\ In the paper, we mainly connect the Gorenstein derived equivalence and stable functors of Gorenstein projective modules.
Specially, we prove that a Gorenstein derived equivalence between CM-finite algebras $A$ and $B$ can induce a stable
functor between the factor categories $A\mbox{-}{\rm mod}/A\mbox{-}{\rm Gproj}$ and $B\mbox{-}{\rm mod}/B\mbox{-}{\rm Gproj}$.
Furthermore, the above stable functor is an equivalence when $A$ and $B$ are Gorenstein.
\end{abstract}

\maketitle


\vskip 20pt

\section{\bf Introduction}

\vskip 5pt

Gorenstein derived categories and Gorenstein derived equivalences were introduced by Gao and Zhang \cite{GZ} have some advantages in the study
of Gorenstein homological algebras. A Gorenstein derived equivalence  is a triangle equivalence between the Gorenstein derived
categories over Artin algebras (see \cite{GZ}). For Gorenstein derived equivalent Artin algebras, it is hard to directly compare the modules over them, since a Gorenstein derived equivalence typically takes modules of one algebra to complexes over the other algebra. It is a well-known result of Rickard \cite{R} which says that a  derived equivalence between two selfinjective algebras always induces a stable
equivalence of Morita type. Recently, Hu and Pan \cite{HP} proved that a derived equivalence induces a stable equivalence of Gorenstein projective objects, where the nonnegative functor and the uniformly bounded functor between derived categories are introduced.

\vskip 5pt

Let $A$ be an Artin algebra and $A\mbox{-}{\rm mod}$ the category of finitely generated $A$-modules.
Let $A\mbox{-}{\rm Gproj}$  be the full subcategory consisting of Gorenstein projective $A$-modules.
The Gorenstein stable category of $A$, denoted by $A\mbox{-}{\rm mod}/A\mbox{-}{\rm Gproj}$, is defined to be
the additive quotient, where the objects are the same as those in $A\mbox{-}{\rm mod}$ and the morphism
space ${\rm Hom}_{A\mbox{-}{\rm mod}/A\mbox{-}{\rm Gproj}}(X, Y)$ is the quotient space of ${\rm Hom}_{A}(X, Y)$
modulo all morphisms factorizing through Gorenstein projective $A$-modules. Two algebras $A$ and $B$ are Gorenstein
stably equivalent if there is an additive equivalence between $A\mbox{-}{\rm mod}/A\mbox{-}{\rm Gproj}$
and $B\mbox{-}{\rm mod}/B\mbox{-}{\rm Gproj}$.

\vskip 5pt

For an arbitrary Gorenstein derived equivalence $F$ between Artin algebras, there is a basic question
when $F$ can induce the Gorenstein stable functor $\overline{F}$ between the corresponding Gorenstein stable categories.

 \vskip 5pt

 In this paper, we consider the nonnegative functor and the uniformly bounded functor between Gorenstein derived categories, respectively.
 We show that two kinds of functors induce the stable functor between the corresponding Gorenstein stable categories (see Proposition~\ref{uniformly}). We prove that  a Gorenstein derived equivalence is both the nonnegative functor and the uniformly bounded functor, and consequently, it can induce a
 Gorenstein stable equivalence (see Theorem~\ref{Gorenstein stable}).

\vskip 5pt

In the following, we recall the basic notions which is necessary in the paper.

\vskip 5pt

From \cite{EJ1, EJ2}, let $A$ be an algebra, $G \in A\mbox{-}{\rm mod}$ is called {\bf Gorenstein projective} if there exists an exact complex $$\dots \lxr P^{-1} \lxr P^{0} \xrightarrow{d^{0}} P^{1} \lxr \dots$$ of projective modules, which stays exact after applying ${\rm Hom}_{A}(-, P)$ for all projective $A$-module $P$, with $G \cong {\rm Ker}d^{0}$.

\vskip 5pt

Denote by $A\mbox{-} \rm{Gproj}$ the full subcategory of $A\mbox{-} {\rm mod}$ consists of all finitely generated Gorenstein projective modules, and $K^{b}(A\mbox{-} \rm{Gproj})$ the full subcategory of $K^{b}(A\mbox{-} {\rm mod})$ consists of all complexes $P^{\bullet}$ with $P^{i} \in A\mbox{-} {\rm Gproj}$ for all $i \in \mathbb{Z}$.

\vskip 5pt

Recall from \cite{BR} and \cite{B} that an algebra $A$ is said to be {\bf CM-finite} if there are only finite indecomposable Gorenstein projective modules in $A\mbox{-}{\rm mod}$, which are $G_{1}, G_{2} ,...,G_{n}$, up to isomorphism. Denote by $G_{A}= \oplus_{i=1}^{n} G_{i}$ the Gorenstein projective generator in $A\mbox{-}{\rm mod}$, and $\aGP(A):=(\rm{End}_{A}(G))^{op}$. Recall from \cite{H} that an Artin algebra $A$ is {\bf Gorenstein} if ${\rm inj.dim} {_{A}A}<\infty$ and ${\rm inj.dim} A_{A}<\infty$.

\vskip 5pt

A complex $X^{\bullet} \in C^{b}(A\mbox{-}{\rm mod})$ is called $G\mbox{-}$acyclic, if ${\rm Hom}_{A}(G, X^{\bullet})$ is acyclic for each $G \in A\mbox{-}{\rm Gproj}$. For a complex $X^{\bullet}\in K^{b}(A\mbox{-} {\rm mod})$, \ $G_{X^{\bullet}} \in K^{-}(A\mbox{-}{\rm Gproj})$ is called a Gorenstein projective resolution of $X^{\bullet}$, \ if there exists a $G\mbox{-}$quasi-isomorphism $f^{\bullet}: G_{X^{\bullet}} \lxr X^{\bullet}$.

\vskip 5pt

Put
$$K^{*}_{gpac}(A\mbox{-}{\rm mod}):=\{X^{\bullet}\in K^{*}(A\mbox{-}{\rm mod})|X^{\bullet} \  is \ G\mbox{-}acyclic\}$$
with $* \in\{blank, -, b\}$.
Then $K^{*}_{gpac}(A\mbox{-}{\rm mod})$ is a thick triangulated subcategory of $K^{*}(A\mbox{-}{\rm mod})$. Following \cite{GZ}, the Verdier quotient
$$D_{gp}^{*}(A\mbox{-}{\rm mod}):= K^{*}(A\mbox{-}{\rm mod})/K^{*}_{gpac}(A\mbox{-}{\rm mod}),$$ which is called the {\bf Gorenstein derived category}. Recall from \cite{GZ} that  two algebras $A$ and $B$ are {\bf Gorenstein derived equivalent}, if there exists a triangle equivalence between $D_{gp}^{b}(A\mbox{-}{\rm mod})$ and $D_{gp}^{b}(B\mbox{-}{\rm mod})$.

\vskip 5pt

Let $A$ and $B$ be CM-finite Gorenstein Artin algebras. Recall from \cite[Remark]{GZ} that $A$ and $B$ are Gorenstein derived equivalent
if and only if there exists a complex $E^{\bullet}\in K^{b}(A\mbox{-}{\rm Gproj})$ such that (1)\ $\aGP(A)\cong {\rm End}_{A}(E^{\bullet})$;
(2)\ ${\rm Hom}_{K^{b}(A\mbox{-}{\rm Gproj})}(E^{\bullet},$ $ E^{\bullet}[i])=0$, $\forall i\neq 0$; and (3)\ ${\rm add}(E^{\bullet})$ generates $K^{b}(A\mbox{-}{\rm Gproj})$ as a triangulated category, where ${\rm add}(E^{\bullet})$ is the full subcategory of $K^{b}(A\mbox{-}{\rm Gproj})$ consisting of direct summands of finite direct sums of $E^{\bullet}$.  In the following, we call $E^{\bullet}$ the Gorenstein silting complex.

\vskip 10pt

\section{\bf Gorenstein stable equivalences}

\vskip 5pt

In this section, we  prove that  a Gorenstein derived equivalence of two CM-finite algebras is both the nonnegative functor and the uniformly bounded functor, and consequently, it can induce a  Gorenstein stable equivalence when above algebras are Gorenstein.

\vskip 10pt

\begin{lem}\
\label{0ncomplex}
Let $A$ and $B$ be two CM-finite algebras such that $G_{A}$ and $G_{B}$ are the Gorenstein projective generators of $A$ and $B$, respectively. Assume that $F: D_{gp}^{b}(A\mbox{-}{\rm mod})\lxr D_{gp}^{b}(B\mbox{-}{\rm mod})$ is the Gorenstein derived equivalence. Then $F(G_{A})$ is isomorphic in $D_{gp}^{b}(B\mbox{-}{\rm mod})$ to a complex $\bar{T}^{\bullet}\in K^{b}(B\mbox{-}{\rm Gproj})$ of the form
$$0\lxr \bar{T}^{0}\lxr \bar{T}^{1}\lxr \cdots\lxr \bar{T}^{n}\lxr 0$$
for some $n\geq 0$ if and only if $F^{-1}(G_{B})$ is isomorphic in $D_{gp}^{b}(A\mbox{-}{\rm mod})$ to a complex $T^{\bullet}\in K^{b}(A\mbox{-}{\rm Gproj})$ of the form
$$0\lxr T^{-n}\lxr \cdots\lxr T^{-1}\lxr T^{0}\lxr 0.$$
\end{lem}
\begin{proof}\ $(\Longrightarrow)$\ Note that
$${\rm Hom}_{K^{b}(A\mbox{-}{\rm Gproj})}(G_{A}, \ T^{\bullet}[i])\xrightarrow{\cong} {\rm Hom}_{K^{b}(B\mbox{-}{\rm Gproj})}(\bar{T}^{\bullet}, \ G_{B}[i])=0$$
for all $i>0$. Then $H^{n}{\rm Hom}_{A}(E, T^{\bullet})=0$ for all $n>0$ and all $E\in A\mbox{-}{\rm Gproj}$. This means that $T^{\bullet}$ splits in all positive degrees, i.e. the exact sequence $$0 \lxr {\rm Ker}(T^{n}\lxr T^{n+1}) \lxr T^{n} \lxr {\rm Im}(T^{n} \lxr T^{n+1}) \lxr 0$$ splits for all positive $n$. Thus we may assume that $T^{i}=0$ for all $i>0$.

\vskip 5pt

To prove that $T^{\bullet}$ is isomorphic to a complex in $K^{b}(A\mbox{-}{\rm Gproj})$ with zero terms in all degrees $<-n$, it suffices to show that ${\rm Hom}_{K^{b}(A\mbox{-}{\rm Gproj})}(T^{\bullet}, E[i])=0$ for all $i>n$ and all $E$ in $A\mbox{-}{\rm Gproj}$. Note that $F(E)\in {\rm add}(\bar{T}^{\bullet})$, which is the smallest full subcategory of $K^{b}(B\mbox{-}{\rm Gproj})$ closed under finite direct sums and direct summands, since $\bar{T}^{\bullet}\cong F(G_{A})$ and $G_{A}$ is the Gorenstein projective generator of $A$. Then we deduce that
$${\rm Hom}_{K^{b}(A\mbox{-}{\rm Gproj})}(T^{\bullet}, \ E[i])\cong {\rm Hom}_{K^{b}(B\mbox{-}{\rm Gproj})}(G_{B}, \ F(E)[i])$$
for all $i>n$.

\vskip 5pt

$(\Longleftarrow)$\ The proof is similar to the above.
\end{proof}

\vskip 10pt

Let $A$ be an Artin algebra, let $Q: K(A\mbox{-}{\rm mod})\lxr D_{gp}(A\mbox{-}{\rm mod})$ be the Verdier quotient functor. Consider the induced map
$$Q_{(X^{\bullet}, \ Y^{\bullet})}: {\rm Hom}_{K(A\mbox{-}{\rm mod})}(X^{\bullet}, \ Y^{\bullet})\lxr {\rm Hom}_{D_{gp}(A\mbox{-}{\rm mod})}(X^{\bullet}, \ Y^{\bullet}).$$

Define

$$\mathcal{U}_{Y^{\bullet}}:=\{X^{\bullet}\in K(A\mbox{-}{\rm mod})\mid Q_{(X^{\bullet}, \ Y^{\bullet}[i])}\ is\ isomorphic\ for\ i\leq 0, $$
$$and\ is \ monic\ for\ i=1\},$$
and for the full subcategory $\mathcal{X}$ of $D_{gp}(A\mbox{-}{\rm mod})$,
$$^{\perp_{G}}\mathcal{X}:= \{Z^{\bullet}\in D_{gp}(A\mbox{-}{\rm mod})\mid {\rm Hom}_{D_{gp}(A\mbox{-}{\rm mod})}(Z^{\bullet}, X^{\bullet}[i])=0 \  for \ i>0 \ and \ X^{\bullet} \in \mathcal{X}\}.$$ $\mathcal{X}^{\perp_{G}}$ is defined dually. Note that for full subcategories $\mathcal{X}$, $\mathcal{Y}$ of triangulated category $\mathcal{C}$, denote by $$\mathcal{X}*\mathcal{Y}:=\{Z \in \mathcal{C}|X \lxr Z \lxr Y \lxr X[1]\ is \ a \ triangle \ in \ \mathcal{C} \ with \ X \in \mathcal{X} \ and \ Y \in \mathcal{Y} \}$$

\vskip 10pt

\begin{lem}
\label{stalk}
Let $A$ be an Artin algebra. Take $X\in A\mbox{-}{\rm mod}$ and a bounded below complex $Y^{\bullet}$ over $A\mbox{-}{\rm mod}$. Suppose that $Y^{i}\in X^{\perp_G}$ for all $i<m$. Then $X[i] \in \mathcal{U}_{Y^{\bullet}}$ for all $i\geq -m$.
\begin{proof}
For $i \geq m$, we have $-m \geq -i$, and $X[-m] \in \mathcal{U}_{Y^i[-i]}$. For  $i<m$, since $Y^i \in X^{\perp_{G}}$, we have $X[-m] \in \mathcal{U}_{Y^i[-i]}$. It follows that $X[-m] \in \mathcal{U}_{Y^i[-i]}$ for all $i \in \mathbb{Z}$. Note that there is some integer $n < m$ such that $Y^i = 0$ for all $i < n$, since $Y^\bullet$ is bounded below. Then $\sigma_{\leq m+1}Y^\bullet$, the left brutal truncation of $Y^\bullet$ at degree m+1, is in $\{Y^{m+1}[-m-1]\} *\ldots *\{Y^n[-n]\}$,  and thus $X[-m] \in \mathcal{U}_{\sigma_{\leq m+1}Y^\bullet}$.

\vskip 5pt

Now it is clear for all $i \leq 1$ that
$${\rm Hom}_{K(A\mbox{-}{\rm mod})}(X[-m], \ (\sigma_{>m+1}Y^\bullet)[i])=0$$
and so $Q_{(X[-m], \ (\sigma_{>m+1}Y^\bullet)[i])}$ is an isomorphism for all $i \leq 1$. This establishes that $X[-m] \in \mathcal{U}_{\sigma _{> m+1}Y^\bullet}$. Since $Y^\bullet$ is in $\{\sigma_{>m+1}Y^\bullet\}*\{\sigma_{\leq m+1}Y^\bullet\}$, we deduce that $X[-m] \in \mathcal{U}_{Y^\bullet}$. Therefore, $X[i] \in \mathcal{U}_{Y^\bullet}$ for all $i\geq -m$.
\end{proof}
\end{lem}

\vskip 10pt

\begin{prop}\
\label{key}
Let $A$ be an Artin algebra and $X^{\bullet}$ and $Y^{\bullet}$  bounded above and bounded below complexes over $A\mbox{-}{\rm mod}$, respectively. Suppose that $X^{i}\in {^{\perp_{G}}Y^{j}}$ for all integers $j<i$. Then $X^{\bullet}\in \mathcal{U}_{Y^{\bullet}}$.
\begin{proof}
We have $X^i[-i] \in \mathcal{U}_{Y^\bullet}$ for all $i \in \mathbb{Z}$ by Lemma~\ref{stalk}. Note that there is an integer $n$ such that $X^i=0$ for all $i > n$, since $X^\bullet$ is bounded above. Thus for each integer $m<n$, the complex $\sigma_{\geq m}X^\bullet$ belongs to $\{X^n[-n]\}*\ldots*\{X^m[-m]\}$, and is consequently in $\mathcal{U}_{Y^\bullet}$. Taking $m$ to be sufficiently small such that $Y^j = 0$ for all $j < m+1$. Then for each integer $i\leq 1$, both ${\rm Hom}_{K(A)}(\sigma_{<m}X^\bullet, \ Y^\bullet[i])$ and ${\rm Hom}_{D_{gp}(A)}(\sigma_{<m}X^\bullet, \ Y^\bullet[i])$ vanish. Hence $Q_{(\sigma_{<m}X^\bullet, \ Y^\bullet[i])}$ is an isomorphism for all $i\leq 1$, and consequently $\sigma_{<m}X^\bullet \in \mathcal{U}_{Y^\bullet}$. Note that $X^\bullet \in \{\sigma_{\geq m} X^\bullet\}*\{\sigma_{<m}X^\bullet\}$. It follows that $X^\bullet \in \mathcal{U}_{Y^\bullet}$.
\end{proof}
\end{prop}

\vskip 10pt

\begin{cor}\
\label{factors}
Let $A$ be an Artin algebra and  let $f: X\lxr Y$ be a homomorphism in $A\mbox{-}{\rm mod}$. Suppose that $Z^{\bullet}$ is a bounded complex over $A\mbox{-}{\rm mod}$ such that $Z^{i}\in X^{\perp_{G}}$ for all $i<0$ and that $Z^{i}\in {^{\perp_{G}}Y}$ for all $i>0$. If $f$ factors through $Z^{\bullet}$ in $D_{gp}^{b}(A\mbox{-}{\rm mod})$, then $f$ factors through $Z^{0}$ in $A\mbox{-}{\rm mod}$.
\begin{proof}
Suppose that $f=gh$ for $g \in {\rm Hom}_{D^{b}_{gp}(A)}(X, Z^{\bullet})$ and $h \in {\rm Hom}_{D^{b}_{gp}(A)}(Z^{\bullet}, Y)$. By Proposition~\ref{key}, both $g$ and $h$ can be presented by a chain map. Namely, $g=g^{\bullet}$ and $h=h^{\bullet}$ in $D^{b}_{gp}(A)$ for some chain maps $g^{\bullet}:X \rightarrow Z^{\bullet}$ and $h^{\bullet}: Z^{\bullet} \rightarrow Y$. Hence $f=g^{\bullet}h^\bullet=g^0 h^0$ in  $D^{b}_{gp}(A)$, and consequently, $f=g^0 h^0$ since $A \hookrightarrow D^{b}_{gp}(A)$ is a fully faithful embedding.
\end{proof}
\end{cor}

\vskip 10pt

The following definition is a Gorenstein version of \cite[Definition 4.1]{HP}.

\vskip 5pt

\begin{defn}\ A triangle functor $F:D_{gp}^{b}(A\mbox{-}{\rm mod})\lxr D_{gp}^{b}(B\mbox{-}{\rm mod})$ is called {\bf uniformly bounded} if there are integers $r<s$ such that $F(X)\in D_{gp}^{[r,s]}(B\mbox{-}{\rm mod})$, i.e. $F(X)^{i}$ vanishes for all $i< r$ and $i> s$, for all $X\in A\mbox{-}{\rm mod}$.
\end{defn}

\vskip 10pt

\begin{defn}
A triangle functor $F:D_{gp}^{b}(A\mbox{-}{\rm mod})\lxr D_{gp}^{b}(B\mbox{-}{\rm mod})$ is called {\bf nonnegative} if it satisfies the following conditions:

\vskip 5pt

\begin{enumerate}
\item \ $H^{i}{\rm Hom}_{D^{b}_{gp}(B\mbox{-}{\rm mod})}(G_{B}, F(X))=0$ for all $i<0$ and $X\in A\mbox{-}{\rm mod}$;

\vskip 5pt

\item \ $F(G)$ is isomorphic to a complex in $K^{b}(B\mbox{-}{\rm Gproj})$ with zero terms in all negative degrees for all $G\in A\mbox{-}{\rm Gproj}$.
\end{enumerate}
\end{defn}

\vskip 10pt

\begin{prop}\
\label{uniformly}
Let $F: D_{gp}^{b}(A\mbox{-}{\rm mod})\lxr D_{gp}^{b}(B\mbox{-}{\rm mod})$ be a Gorenstein derived equivalence between two CM-finite algebras $A$ and $B$. Then

\begin{enumerate}
\item \ $F$ is uniformly bounded.

\vskip 5pt

\item \ $F$ is nonnegative if and only if the Gorenstein tilting complex associated to $F$ is isomorphic in $K^{b}(A\mbox{-}{\rm Gproj})$ to a complex with zero terms in all positive degrees.
\end{enumerate}
\end{prop}
\begin{proof}\  By \cite[Proposition 4.2]{GZ} we see that the functor $F$ induces a triangle equivalence  between $K^{b}(A\mbox{-}{\rm Gproj})$ and $K^{b}(B\mbox{-}{\rm Gproj})$. By Lemma~\ref{0ncomplex}, we may assume that $E^{\bullet}$ is a complex associated to $F$ such that $E^{i}=0$ for all $i>0$ and $i<-n$ and $F(E^{\bullet})\cong G_{B}$.

\vskip 5pt

(1) \  Let $X$ be in $A\mbox{-}{\rm mod}$. Then for any integer $i$,
$$\begin{aligned}
H^{i}({\rm Hom}_{B}(G_{B}, F(X)))&\cong{\rm Hom}_{D^{b}_{gp}(B\mbox{-}{\rm mod})}(G_{B}, F(X)[i])\\
&\cong {\rm Hom}_{K^{b}(A\mbox{-}{\rm mod})}(E^{\bullet}, X[i]).
\end{aligned}$$
So $H^{i}({\rm Hom}_{B}(G_{B}, F(X)))=0$ for all $i>n$ and $i<0$. This proves that $F$ is uniformly bounded.

\vskip 5pt

(2) \ By Lemma ~\ref{0ncomplex}, $F(G_{A})$ is isomorphic to a complex $\bar{E}^{\bullet}\in K^{[0,n]}(B\mbox{-}{\rm Gproj})$ for some nonnegative integer $n$. As an equivalence, $F$ preserves coproducts. This means that $F(G)\subseteq K^{[0,n]}(B\mbox{-}{\rm Gproj})$ for any $G\in A\mbox{-}{\rm Gproj}$. Let $X$ be an $A\mbox{-}$module. Since
$${\rm Hom}_{D_{gp}^{b}(B\mbox{-}{\rm mod})}(G_{B},\ F(X)[i])\cong {\rm Hom}_{D_{gp}^{b}(A\mbox{-}{\rm mod})}(E^{\bullet},\ X[i])=0$$
for all $i<0$, it follows that $H^{i}({\rm Hom}_{B}(G_{B}, F(X)))=0$ for all $X\in A\mbox{-}{\rm mod}$ and all $i<0$, i.e., $F$ is nonnegative.

\vskip 5pt

Conversely, suppose that $F$ is a nonnegative Gorenstein derived equivalence. Then $F(G_{A})$ is isomorphic to a bounded complex $Q^{\bullet}\in K^{\geq 0}(B\mbox{-}{\rm Gproj})$. Note that for all $i>0$,
$${\rm Hom}_{K^{b}(A\mbox{-}{\rm Gproj})}(G_{A},\ E^{\bullet}[i])\cong {\rm Hom}_{K^{b}(B\mbox{-}{\rm Gproj})}(F(G_{A}),\ G_{B}[i])=0.$$
This means that $H^{i}{\rm Hom}_{A}(G, \ E^{\bullet})=0$ for all $G\in A\mbox{-}{\rm Gproj}$ and $i>0$. This shows that $E^{\bullet}$ splits in all positive degrees and thus is isomorphic to a complex in $K^{b}(A\mbox{-}{\rm Gproj})$ with zero terms in all positive degrees.
\end{proof}

\vskip 10pt

\begin{lem}
\label{nonnegative}
Let $F: D^{b}_{gp}(A\mbox{-}{\rm mod}) \rightarrow D^{b}_{gp}(B\mbox{-}{\rm mod})$ be a uniformly bounded, nonnegative triangle functor. Suppose that there exists an integer $n>0$ such that $F(X) \subseteq D^{[0, n]}_{gp}(B\mbox{-}{\rm mod})$ for any $X\in A\mbox{-}{\rm mod}$. Then the following statements hold.

\vskip 5pt

\begin{enumerate}
\item \ If $F$ admits a right adjoint $H$, then $H$ is uniformly bounded and $H(Y) \subseteq D_{gp}^{[-n,0]}(A\mbox{-}{\rm mod})$.

\vskip 5pt

\item \ If $F$ admits a left adjoint $L$, then $L(Q)\in D^{[-n,0]}_{gp}(A\mbox{-}{\rm mod})$  for any $Q\in B\mbox{-}{\rm Gproj}$.

\vskip 5pt

\item \ If $H$ is both a left and right adjoint of $F$, then $H[-n]$ is uniformly bounded and nonnegative.
\end{enumerate}
\begin{proof}
(1) \ Let $P$ be a Gorenstein projective $A\mbox{-}$module. Then for any $B\mbox{-}$module $Y$,
$${\rm Hom}_{D^{b}_{gp}(A\mbox{-}{\rm mod})}(P, \ H(Y)[i]) \cong {\rm Hom}_{D^{b}_{gp}(B\mbox{-}{\rm mod})}(F(P), \ Y[i])=0$$
for all $i> 0$ and $i<-n$ since $F(P) \in K^{[0, n]}(B\mbox{-}{\rm Gproj})$. It follows that $H(Y) \in D^{[-n, 0]}_{gp}(A\mbox{-}{\rm mod})$.

\vskip 5pt

(2) \ Let $Q \in B\mbox{-}{\rm Gproj}$ and $X\in A\mbox{-}{\rm mod}$. Then
$${\rm Hom}_{D^{b}_{gp}(A\mbox{-}{\rm mod})}(L(Q), X[i])\cong {\rm Hom}_{D^{b}_{gp}(B\mbox{-}{\rm mod})}(Q, F(X)[i])=0$$
for all $i>n$ and $i<0$. This implies that $L(Q)\in K^{[-n, 0]}(A\mbox{-}{\rm Gproj})$.

\vskip 5pt

(3) \ It follows from (1) and (2) immediately.

\end{proof}
\end{lem}

\vskip 10pt

\begin{lem}
Let $F: D^{b}_{gp}(A\mbox{-}{\rm mod}) \rightarrow D^{b}_{gp}(B\mbox{-}{\rm mod})$ be a uniformly bounded, nonnegative triangle functor. Then for any $A\mbox{-}$module $X$, there is a triangle $$U^{\bullet}_{X} \stackrel{i_{X}}{\longrightarrow}F(X) \stackrel{\pi_{X}}{\longrightarrow} M_{X} \stackrel{\mu_{X}}{\longrightarrow} U^{\bullet}_{X}[1]$$ in $D^{b}_{gp}(B\mbox{-}{\rm mod})$ with $M_{X} \in B\mbox{-}{\rm mod}$ and $U^{\bullet}_{X} \in D^{[1, n_{X}]}_{gp}(B\mbox{-}{\rm Gproj})$ for some $n_{X}>0$.
\end{lem}

\vskip 5pt

\begin{proof}
Let $U^{\bullet}_{X}$ be the Gorenstein projective resolution of $F(X)$, and then do good truncation of degree zero.
\end{proof}

\vskip 10pt

\begin{lem}
Assume that $U^{\bullet}_{i} \stackrel{\alpha_{i}}{\longrightarrow} X^{\bullet}_{i} \stackrel{\beta_{i}}{\longrightarrow} M_{i} \stackrel{\gamma_{i}}{\longrightarrow} U^{\bullet}_{i}[1]$, i=1, 2, are triangles in $D^{b}_{gp}(B\mbox{-}{\rm mod})$ such that $M_{1}, M_{2}$ are in $B\mbox{-}{\rm mod}$ and $U^{\bullet}_{1}, U^{\bullet}_{2} \in K^{[1, n]}(B\mbox{-}{\rm Gproj})$. Then, for each morphism $f: X^{\bullet}_{1} \rightarrow X^{\bullet}_{2}$ in $D^{b}_{gp}(B\mbox{-}{\rm mod})$, there is a morphism $b: M_{1} \rightarrow M_{2}$ in $B\mbox{-}{\rm mod}$ and a morphism $a: U^{\bullet}_{1} \rightarrow U^{\bullet}_{2}$ in $D^{b}_{gp}(B\mbox{-}{\rm mod})$ such that the diagram
$$\begin{CD}
U^{\bullet}_{1} @>\alpha_{1}>> X^{\bullet}_{1} @>\beta_{1}>> M_{1} @>\gamma_{1}>> U^{\bullet}_{1}[1]\\
@VVaV  @VVfV  @VVbV  @VVa[1]V\\
U^{\bullet}_{2} @>\alpha_{2}>> X^{\bullet}_{2} @>\beta_{2}>> M_{2} @>\gamma_{2}>> U^{\bullet}_{2}[1]
\end{CD}$$ commutes. Moreover, if $f$ is an isomorphism in $D^{b}_{gp}(B\mbox{-}{\rm mod})$, then $\underline{b}$ is an isomorphism in $B\mbox{-}{\rm mod}/B\mbox{-}{\rm Gproj}$.
\begin{proof}
Since
$${\rm Hom}_{D^{b}_{gp}(B\mbox{-}{\rm mod})}(U^{\bullet}_{1}, M_{2})\cong {\rm Hom}_{K^{b}(B\mbox{-}{\rm mod})}(U^{\bullet}_{1}, M_{2})=0,$$
we see that $\alpha_{1}f\beta_{2}$ is zero, and so $a$ and $b$ exist.

\vskip 5pt

Now assume that $f$ is an isomorphism in $D^{b}_{gp}(B\mbox{-}{\rm mod})$. Namely, there is an isomorphism $g: X^{\bullet}_{2} \rightarrow X^{\bullet}_{1}$ in $D^{b}_{gp}(B\mbox{-}{\rm mod})$ such that $fg={\rm Id}_{X^{\bullet}_{1}}$ and $gf={\rm Id}_{X^{\bullet}_{2}}$. By the above similar discussion, there is a morphism $c: M_{2} \rightarrow M_{1}$, such that $c\beta_{2}=\beta_{1}g$. Then $$\beta_{1}-cb\beta_{1}=\beta_{1}-c\beta_{2}f=\beta_{1}-\beta_{1}gf=0$$
and ${\rm Id}_{M_{1}}-cb$ factors through $U^{\bullet}_{1}[1]$. Then ${\rm Id}_{M_{1}}-cb$ factors through a Gorenstein projective $B\mbox{-}$module, and hence $\underline{c}\circ \underline{b}=\underline{{\rm Id}}_{M_{1}}$  in $B\mbox{-}{\rm mod}/B\mbox{-}{\rm Gproj}$. Similarly we have $\underline{b}\circ\underline{c}=\underline{{\rm Id}}_{M_{2}}$, and therefore $\underline{b}: M_{1} \rightarrow M_{2}$ is an isomorphism in $B\mbox{-}{\rm mod}/B\mbox{-}{\rm Gproj}$.
\end{proof}
\end{lem}

\vskip 10pt

Next we define the functor $\overline{F}: A\mbox{-}{\rm mod}/A\mbox{-}{\rm Gproj}\rightarrow B\mbox{-}{\rm mod}/B\mbox{-}{\rm Gproj}$ as follows. We fix a triangle
$$\xi_{X}: U^{\bullet}_{X} \stackrel{i_{X}}\longrightarrow F(X) \stackrel{\pi_{X}}\longrightarrow M_{X} \stackrel{\mu_{X}}\longrightarrow U^{\bullet}_{X}[1]$$ in $D^{b}_{gp}({B\mbox{-}{\rm mod}})$ with $M_{X}\in B\mbox{-}{\rm mod}$ and $U^{\bullet}_{X}$ a complex in $K^{[1, n_{X}]}(B\mbox{-}{\rm Gproj})$ for some $n_{X}>0$.  For each morphism $f: X\rightarrow Y$ in $A\mbox{-}{\rm mod}$, we can form the following diagram in $D^{b}_{gp}({B\mbox{-}{\rm mod}})$:
$$
\begin{CD}
U^{\bullet}_{X} @>i_{X}>> F(X) @>\pi_{X}>> M_{X} @>\mu_{X}>> U^{\bullet}_{X}[1]\\
 @VVa_{f}V                @VVF(f)V         @VVb_{f}V            @VVa_{f}[1]V\\
U^{\bullet}_{Y} @>i_{Y}>> F(Y) @>\pi_{Y}>> M_{Y} @>\mu_{Y}>> U^{\bullet}_{Y}[1]\\
\end{CD}
$$
If $b^{'}_{f}$ is another morphism such that $b^{'}_{f}\pi_{X}=\pi_{Y}F(f)$, then $(b_{f}-b^{'}_{f})\pi_{X}=0$ and $b_{f}-b^{'}_{f}$ factors through $U^{1}_{X}$ which is Gorenstein projective, so $\underline{b_{f}}=\underline{b^{'}_{f}} \in {\rm Hom}_{B\mbox{-}{\rm mod}/B\mbox{-}{\rm Gproj}}(M_{X}, \ M_{Y})$. Moreover, if $f$ factors through a Gorenstein projective $A$-module $P$, say $f=hg$ for $g:X\longrightarrow P$ and $h: P\longrightarrow Y$. Then $(b_{f}-b_{h}b_{g})\pi_{X}=\pi_{Y}F(f)-\pi_{Y}F(h)F(g)=0$. Hence $b_{f}-b_{h}b_{g}$ factors through $U^{\bullet}_{X}[1]$,  and then $b_{f}$ factors through $P\oplus U^{1}_{X}$ which is Gorenstein projective. Hence $\underline{b_{f}}=0$. Thus we get a well-defined map:
$$\phi: {\rm Hom}_{A\mbox{-}{\rm mod}/A\mbox{-}{\rm Gproj}}(X, \ Y) \rightarrow {\rm Hom}_{B\mbox{-}{\rm mod}/B\mbox{-}{\rm Gproj}}(M_{X}, \ M_{Y}), \ \underline{f} \mapsto \underline{b_{f}}.$$
Define $\overline{F}:= M_{X}$ for each $X \in A\mbox{-}{\rm mod}/A\mbox{-}{\rm Gproj}$ and $\overline{F}(\underline{f}):= \phi(\underline{f})$ for each morphism $f$ in $A\mbox{-}{\rm mod}$, we get a functor $\overline{F}: A\mbox{-}{\rm mod}/A\mbox{-}{\rm Gproj}\longrightarrow B\mbox{-}{\rm mod}/B\mbox{-}{\rm Gproj}$.

\vskip 10pt

\begin{thm}
\label{Gorenstein stable}
Let $A$ and $B$ be CM-finite algebras, and $F:D^{b}_{gp}(A\mbox{-}{\rm mod}) \rightarrow D^{b}_{gp}(B\mbox{-}{\rm mod})$ be a Gorenstein derived equivalence with the quasi-inverse $G$. Then the induced stable functors $$\overline{F}: A\mbox{-}{\rm mod}/A\mbox{-}{\rm Gproj}\longrightarrow B\mbox{-}{\rm mod}/B\mbox{-}{\rm Gproj}$$ and $$\overline{G}: B\mbox{-}{\rm mod}/B\mbox{-}{\rm Gproj}\rightarrow A\mbox{-}{\rm mod}/A\mbox{-}{\rm Gproj}$$
satisfy
$$\overline{F} \circ \overline{G[-n]} \cong \overline{[-n]} \cong \Omega^{n}_B, \ and \ \overline{G[-n]}\circ\overline{F}  \cong \overline{[-n]} \cong \Omega^{n}_A,$$
where $\Omega^{n}_A$ and $\Omega^{n}_B$ are the $n$-th syzygy functor.

\vskip 5pt

Consequently, if furthermore $A$ and $B$ are Gorenstein algebras, then $A$ and $B$ are Gorenstein stably equivalent.
\begin{proof}\
Since $G$ is both a left adjoint and right adjoint of $F$, there exists some integer $n>0$ such that $G[-n]$ is nonnegative by Lemma~\ref{nonnegative}. Moreover, $F[n]$ is a right adjoint of $G[-n]$, and sends Gorenstein projective $A\mbox{-}$modules to complexes in $K^{b}(B\mbox{-}{\rm Gproj})$.

\vskip 5pt

Note that there is the  following diagram of functors as follows:
$$
\begin{CD}
A\mbox{-}{\rm mod}/A\mbox{-}{\rm Gproj} @>can>> D^{b}_{gp}(A\mbox{-}{\rm mod})/K^{b}(A\mbox{-}{\rm Gproj})\\
@VV\overline{F}V  @VVFV \\
B\mbox{-}{\rm mod}/B\mbox{-}{\rm Gproj} @>can>> D^{b}_{gp}(B\mbox{-}{\rm mod})/K^{b}(B\mbox{-}{\rm Gproj})\\
\end{CD}
$$
Thus, we have isomorphisms of functors
$$\overline{F} \circ \overline{G[-n]} \cong \overline{[-n]} \cong \Omega^{n}_B$$
and
$$ \overline{G[-n]}\circ\overline{F}  \cong  \overline{[-n]} \cong \Omega^{n}_A.$$

\vskip 5pt

If $A$ and $B$ are Gorenstein, then from above equalities $A$ and $B$ are Gorenstein stably equivalent.
\end{proof}
\end{thm}


\vskip 10pt



\vskip 20pt

\begin{thebibliography}{99}
\bibitem{B} A.Beligiannis. Cohen-Macauley Modules, (Co)Torsion Pairs, and Virtually Gorenstein Algebras, J. Algebra, 288(2005), 137-211.
\bibitem{BR} A.Beligiannis, I.Reiten, Homological and homotopical aspects of torsion theories, Mem. A mer. Math. Soc. 188(2007), no. 883, \romannumeral8 +207 pp.
\bibitem{GZ} N.Gao, P.Zhang, Gorenstein derived categories, J. Algebra 323(2010), 2041-2057.
\bibitem{HP} W.Hu, S.Y.Pan, Stable functors of derived equivalences and Gorenstein projective modules, Math. Nachr. 290(2017), 1512-1530.
\bibitem{EJ1} E.E.Enochs, O.M.G.Jenda, Gorenstein injective and projective modules, Math.Z. 220(1995), 611-633.
\bibitem{EJ2} E.E.Enochs, O.M.G.Jenda, Relative Homological Algebra, de Gruyter Exp. Math., col.30, Walter de Gruyter and Co., 2000.
\bibitem{H} D.Happel, On Gorenstein algebras, In: Representation Theory of Finite Groups and Finite-Dimensional Algebras, Progress in Mathematics, vol. 95. Basel: Birkh\"{a}user, 1991, 389-404.
\bibitem{R} J.Rickard, Derived equivalences as derived functors, J. London Math. Soc. 43(1991), 37-48.
\end{thebibliography}
\end{document}